\documentclass[a4paper,10pt,reqno, english]{amsart}

\usepackage{amsmath,amssymb,amscd,amsthm,amsfonts}
\usepackage{graphicx,subcaption}
\usepackage{hyperref}
\usepackage{dsfont}
\usepackage{xcolor}
\usepackage{color,soul}
\usepackage[utf8]{inputenc}
\usepackage[nobysame, alphabetic]{amsrefs}
\usepackage[capitalize]{cleveref}
\usepackage{float}

\newtheorem{theorem}{Theorem}[section]
\newtheorem{corollary}[theorem]{Corollary}
\newtheorem{lemma}[theorem]{Lemma}

\newtheorem{proposition}[theorem]{Proposition}

\theoremstyle{definition}
\newtheorem{remark}[theorem]{Remark}
\newtheorem{definition}[theorem]{Definition}

\DeclareMathOperator{\bbox}{box}
\newcommand{\R}{\mathds{R}}
\newcommand{\rr}{\mathds{R}}

\newcommand{\cal}[1]{\mathcal{#1}}

\title{Extensions of discrete Helly theorems for boxes}


\author[Edwards]{Timothy Edwards}\address{Gannon University, Erie, PA 16541} 
\email{edwards096@gannon.edu}\thanks{The research of T. Edwards was supported by NSF grant DMS 2051026.}

\author[Sober\'on]{Pablo Sober\'on}\address{Baruch College and The Graduate Center, City University of New York, New York, NY 10010} \thanks{The research of P. Sober\'on was supported by NSF CAREER award no. 2237324, NSF award no. 2054419 and a PSC-CUNY Trad B award.} 
\email{psoberon@gc.cuny.edu}

\begin{document}

\begin{abstract}
We prove extensions of Halman's discrete Helly theorem for axis-parallel boxes in $\rr^d$.  Halman's theorem says that, given a set $S$ in $\rr^d$, if $F$ is a finite family of axis-parallel boxes such that the intersection of any $2d$ contains a point of $S$, then the intersection of $F$ contains a point of $S$.  We prove colorful, fractional, and quantitative versions of Halman's theorem.  For the fractional versions, it is enough to check that many $(d+1)$-tuples of the family contain points of $S$.  Among the colorful versions we include variants where the coloring condition is replaced by an arbitrary matroid.  Our results generalize beyond axis-parallel boxes to $H$-convex sets.
\end{abstract}

\maketitle

\section{Introduction}

The study of intersection patterns of convex sets in Euclidean spaces is a major area in combinatorial geometry.  One of the central results is Helly's theorem, which characterizes finite families of convex sets in $\rr^d$ with non-empty intersection \cite{Helly:1923wr}.  It states that \textit{if every $d+1$ or fewer sets of a finite family of convex sets in $\rr^d$ have non-empty intersection, then so does the entire family.}  There are now a myriad generalizations and extensions of Helly's theorem.  For example, we might weaken the condition on intersections of small subfamilies, we might want to guarantee that the intersection of the entire family is larger, or we might want to replace the condition of convexity by another geometric or topological constraint on the sets \cites{Amenta:2017ed, Holmsen:2017uf, Barany2022}.

A particular family of sets for which there are Helly-type theorems are axis-parallel boxes \cites{Eckhoff:1988eo, Eckhoff:1991jj}.  As these are sets which are much more constrained, we expect to be able to prove stronger results on their intersection structure.  A simple folklore result, for example, states that \textit{a finite family of axis parallel-boxes has a non-empty intersection if an only if every pair of boxes has non-empty intersection.}  Helly-type results of this kind don't give us much information about the points of intersection of the whole family.  However, a recent result of Halman \cite{Halman2008}, the discrete Helly theorem for axis-parallel boxes, gives insight in this direction.

\begin{theorem}[Halman 2008]\label{thm:halman}
    Let $d$ be a positive integer.  Let $S \subset \rr^d$ be a finite set, and $\mathcal{F}$ be a finite family of axis-parallel boxes in $\rr^d$.  If every $2d$ boxes have a point of $S$ in their intersection, then $\bigcap \mathcal{F}$ has a point of $S$.
\end{theorem}

The goal of this manuscript is to extend Halman's result in several directions in which Helly's theorem has been generalized.  Notice that the Helly number, the size of the subfamilies we need to check, grows from $2$ to $2d$ once we have the set $S$ involved and $2d$ is optimal.  One feature of Helly's theorem for general convex sets is that the important parameter in a vast number of generalizations is always $d+1$.  This can be seen in the fractional and colorful versions of this result.  In contrast, many of the generalizations of Halman's result will have different Helly numbers. These results provide a concrete example that exhibits how different versions of Helly's theorem relate to each other.  All the results in this manuscript generalize to $H$-convex sets, which are a family of convex sets that contains axis-parallel boxes \cite{Boltyanski:2003ir}.  As the literature for axis-parallel boxes is much more extensive, we first state and prove our results for axis-parallel boxes, and in \cref{sec:hconvex} we describe their generalizations.

Our first result is a quantitative colorful version of Halman's result.  We prove the result for several families of axis-parallel boxes.  Moreover, we can also impose how many points of $S$ we want in the conclusion of the theorem.

\begin{theorem}\label{thm: colorful quantitative halman}
    Let $\cal{B}_1,\cal{B}_2,\cdots,\cal{B}_{2d}$ be finite families of axis parallel boxes in $\R^d$. Let $S\subseteq\rr^d$ be a discrete set and $n$ be a positive integer. Suppose that for every choice of  $B_i\in\cal{B}_i$ for each $i \in [2d]$ we have $\left| \bigcap_{i=1}^{2d}B_i\cap S\right|\geq n$.
    Then there exists an index $l\in[n]$ such that the intersection of $\cal{B}_l$ contains al least $n$ points of $S$.
\end{theorem}

Statements like this are known as ``colorful'' theorems since we can consider all sets in a family $\cal{B}_i$ as being drawn with a particular color.  The main gist is that hypotheses of the colorful $2d$-tuples imply results on a monochromatic family.  Quantitative discrete Helly theorems for convex sets require conditions on the set $S$.  Additionally, the quantitative Helly number often depends on the number $n$ of $S$ we guarantee in the intersection, while for axis-parallel boxes this is not the case (see, e.g., \cites{Aliev:2014kh, DeLoera:2017bl, Chestnut2018}).  For general convex sets, the dependence of the Helly number on $n$ can be removed if we impose additional structural conditions on the set of $n$ points, such as being collinear \cite{Dillon2021}.

Throughout the paper, we prove several other colorful versions of Halman's theorem. 
 A ``very colorful'' version of the theorem above (\cref{thm:very-colorful-quantitative-halman}), a version where the condition is replaced by a matroid (\cref{thm:matroid-halman}), and a version where the number of color classes is reduced to $d+1$ with a weaker conclusion (\cref{thm:colorful-halman-few-colors}).  We also present two fractional versions of Halman's theorem.  The first one, \cref{thm:fractional-strong}, says that if and $\alpha$-fraction of $2d$-tuples of $\cal B$ contain a point of $S$ in their intersection, then there is a $\beta$-fraction of $\cal B$ whose intersection contains points of $S$, and $\beta \to 1$ as $\alpha \to 1$.  This shows a robustness of Halman's theorem.

The second fractional version is perhaps the most interesting one, as we need to check smaller subfamilies.

\begin{theorem}\label{thm:fractional-small}
    Let $d$ be a positive integer and $\alpha \in (0,1)$.  There exists $\beta = \beta(\alpha, d) > 0$ such that the following holds.  If $S$ is a finite set in $\rr^d$ and $\cal B$ is a finite family of axis-parallel boxes of $\rr^d$ such that $\alpha \binom{|\cal B|}{d+1}$ of the $(d+1)$-tuples of $\cal{B}$ satisfy that their intersection contains points of $S$, then there is a subfamily of at least $\beta |\cal B|$ sets of $\cal B$ whose intersection contains a point of $S$.
\end{theorem}

Of course, in this theorem we cannot expect $\beta \to 1$ as $\alpha \to 1$ since Halman's theorem is optimal.  It mimics the behavior of Helly's theorem for integer points.  If we want to check if a family of convex sets has an integer point in its intersection, it's sufficient to check the intersection of its $2^d$-tuples \cites{Doignon:1973ht, Scarf:1977va, Bell:1977tm}, and $2^d$ is optimal.  However, if we want to check that there is an integer point in a positive fraction of the family, it's enough to check that a positive fraction of the $(d+1)$-tuples have a point in their intersection \cite{Anonymous:PHt9HPGF}.  This result extends to families of points with finite Helly number \cite{Averkov2012}, for volumetric versions of Helly's theorem \cite{Frankl2024}, or set systems with topological conditions \cite{Goaoc2021}.

The size of the sets we need to check to guarantee a fractional Halman is neither the Helly number for axis parallel boxes (two), nor Halman's number ($2d$).  The reason that $d+1$ appears here is because it is the number for a colorful Helly for axis-parallel boxes.  It is not enough to check $d$-tuples.  Given positive integers $n,d$, we can take $S = \rr^d$ and make a family of axis-parallel boxes where $\sim n/d$ of the boxes are the intersection of a hyperplane orthogonal to $e_i$, the $i$-th element of the canonical basis, with the unit square $[0,1]^d$.  Then, we have roughly $(n/d)^d \sim (d!/d^d) \binom{n}{d}$ intersecting $d$-tuples, but no subset of size $d+1$ or lager intersects.  We can also make $S$ finite while preserving the condition of the theorem.

As fractional Helly theorems are known to imply $(p,q)$ variants \cites{Alon:1992ta, Alon:1996uf}, we obtain the following $(p,q)$ generalization of Halman's theorem as a direct consequence of \cref{thm:fractional-small}.

\begin{theorem}
    Let $p,q,d$ be positive integers such that $p \ge q \ge d+1$.  There exists a positive integer $c=c(p,q,d)$ such that the following holds.  For any set $S \subset \rr^d$ and any finite family $\cal{B}$ of axis-parallel boxes, if out of every $p$ boxes in $\cal{B}$ there are $q$ whose intersection contains a point of $S$, then there is a set of $c$ points of $S$ that pierces every set in $\cal{B}$.
\end{theorem}

The rest of the paper is organized as follows.

\begin{itemize}
    \item In \cref{sec:new-proofs} we prove the colorful quantitative version of Halman's theorem and the ``very colorful'' version.  We also give another proof of Halman's theorem using an argument similar to Radon's theorem.
    \item In \cref{sec:topological} we introduce a family of simplicial complexes arising from Halman's theorem and prove that they have bounded Leray numbers.  With this, we prove the matroid colorful Halman theorem and the our bounds for the fractional Halman theorem when we check $2d$-tuples.
    \item In \cref{sec:fractional-and-beyond} we prove the fractional version of Halman's theorem for $(d+1)$-tuples.  We use a hypergraph saturation argument along with the colorful version of Halman's theorem for $d+1$ colors.
    \item Finally, in \cref{sec:hconvex} we introduce $H$-convex sets and state how the results of this manuscript generalize to those sets.
\end{itemize}

\section{New proofs of Halman's theorem}\label{sec:new-proofs}

\subsection{Colorful and quantitative Halman}

Let us start by proving \cref{thm: colorful quantitative halman}.  For this, we introduce the lexicographic order $\prec$ on $\rr^d$.  Given two different points $x = (x_1, \dots, x_d)$ and $ y = (y_1, \dots, y_d) $ in $\rr^d$, we say $x \prec y$ if there exists an $i \in [d]$ such that $x_j = y_j$ for all $j < i$ and $x_i < y_i$.  In other words, in the first coordinate in which the points are different, $x$ has the smaller value.


\begin{proof}[Proof of \cref{thm: colorful quantitative halman}]
     We call a subfamily $\cal{F}\subseteq\cal{B}_1\cup\cdots\cup\cal{B}_{2d}$ colorful if no two boxes in $\cal{F}$ come from the same indexed set $\cal{B}_i$. For each $k \in [n]$, consider a function $f_k$ from the set of all colorful $(2d-1)$-tuples to the set $S$. The function $f_k$ assigns to each colorful $(2d-1)$-tuple the $k$-th smallest point of $S$, under the lexicographic order, which is contained in the intersection of this tuple. As $S$ is discrete and each colorful $(2d-1)$-tuple must contain at least $n$ points of $S$, the function $f_k$ is well defined.  Since there is a finite number of possible $(2d-1)$-tuples, there is a $(2d-1)$-tuple $\cal{A}$ on which $f_n$ attains its maximal value (again, under the lexicographic order).
     
 Without loss of generality, assume that $\cal{A}$ does not have a set of $\cal{B}_{2d}$ and $\cal{A}=\{A_1,\cdots,A_{2d-1}\}$ where $A_i\in\cal{B}_i$ for each $i \in [2d-1]$. Consider a set $B\in\cal{B}_{2d}$. We claim that $f_1(\cal{A}),\dots,f_n(\cal{A})\in B$. The intersection $(\cap\cal{A})\cap B$ is also an axis-parallel box, so its projection to the $i$-th coordinate is an interval, which we denote by $[b_i, c_i]$. Let $B_i$ be a box in $\cal{A}\cup B$ which achieves $b_i$ as a left endpoint when projected to the $i$-th coordinate. Similarly, let $C_i$ be a box from this collection which achieves $c_i$ as a right endpoint when projected to the $i$-th coordinate. 
 
 Consider the subfamily $\cal{A}'=\{B_1,\dots,B_d,C_2,\cdots,C_d\}$ (note we are not using $C_1$). It may be the case that these sets are not all unique and hence $|\cal{A}'|<2d-1$. If this is the case, include additional arbitrary boxes from $\cal{A}\cup B$ into $\cal{A}'$ so that we have $|\cal{A}'|=2d-1$. By construction, $\cal{A}'$ is a colorful $(2d-1)$-tuple, and $\cap\cal{A}'$ contains only points of $S$ which are either in $(\cap\cal{A})\cap B$ or are lexicographically greater than all the points in $(\cap\cal{A})\cap B$, since their first coordinate must be larger than $c_1$. In particular, we have $f_i(\cap \cal{A}')=f_i((\cap\cal{A})\cap B)$ for all $i\in [n]$. All points of $S$ in $\cap\cal{A}\cap B$ are also in $\cap\cal{A}$. In particular $\cap\cal{A}\cap B$ must contain the first $n$ lexicographic points of $S$ of $\cap\cal{A}$, or else $f_n(\cal{A}')>f_n(\cal{A})$ which contradicts the maximality of $f_n(\cal{A})$. This implies that $f_1(\cal{A}),\cdots,f_n(\cal{A})\in B$ and therefore $f_1(\cal{A}),\cdots,f_n(\cal{A})\in \cap\cal{B}_{2d}$. Thus $|\cap\cal{B}_{2d}\cap S|\geq n$.
\end{proof}

Theorem \ref{thm: colorful quantitative halman} is both a colorful and quantitative in nature.
By setting all the color classes equal to each other $\cal{B}_1=\cal{B}_2=\cdots=\cal{B}_{2d}$ we obtain the following monochromatic version.

\begin{corollary}
    Given a family $\cal{B}\subseteq\R^d$ of axis parallel boxes and $S\subseteq\R^d$ a discrete set, if every choice of $2d$ boxes from $\cal{B}$ contain $n$ points of $S$ in their intersection, then $\cap\cal{B}$ contains $n$ points of $S$.
\end{corollary}

This quantitative result is an extension of a result by Halman who proved the $n=1$ case in \cite{Halman2008}. Our work provides an alternate proof of Halman's original theorem.

The version above also has the following implication for quantitative Helly for boxes.

\begin{corollary}\label{cor:quantitative-halman}
    Let $\mu$ be a finite measure in $\rr^d$, absolutely continuous with respect to the Lebesgue measure.  Let $\cal{B}$ be a finite family of axis-parallel boxes.  If the intersection of any $2d$ boxes of $\mathcal{B}$ has $\mu$-measure at least $1$, then $\mu\left(\bigcap \cal{B}\right) \ge 1$.
\end{corollary}

\begin{proof}
    As any measure can be approximated by finite sets of points, consider a sequence $S_n$ of finite sets of $\rr^d$ such that for each $\cal{B}' \subset \cal{B}$ we have
    \[
\lim_{n \to \infty} \frac{\left|S_n \cap (\bigcap \cal{B}')\right|}{|S_n|} \to \frac{\mu(\cal{B}')}{\mu(\rr^d)} .
    \]

    Let $\varepsilon>0$.  There exists a natural number $N$ such that for $n>N$ we have that for every $2d$-tuple $\cal{B}'\subset \cal{B}$ we have 
\begin{align*}
\frac{\left|S_n \cap (\bigcap \cal{B}')\right|}{|S_n|} & > \frac{\mu(\cap \cal{B}')}{\mu(\rr^d)} - \varepsilon \ge \frac{1}{\mu(\rr^d)} - \varepsilon \\
\left|S_n \cap (\bigcap \cal{B}')\right| & > |S_n| \dot {\left( {\frac{1}{\mu(\rr^d)}} - \varepsilon \right)}
\end{align*}

Applying \cref{cor:quantitative-halman}, we have that 
\begin{align*}
    \left|S_n \cap (\bigcap \cal{B})\right| & \ge |S_n| \dot {\left( {\frac{1}{\mu(\rr^d)}} - \varepsilon \right)} \\
    \frac{\left|S_n \cap (\bigcap \cal{B})\right|}{|S_n|} & \ge {\frac{1}{\mu(\rr^d)}} - \varepsilon.
\end{align*}

As $n \to \infty$, this implies that 
\[
\frac{\mu(\cap \cal{B})}{\mu(\rr^d)} \ge \frac{1}{\mu(\rr^d)},
\]
giving us the conclusion we wanted.
\end{proof}

This corollary is also a direct consequence of the quantitative Helly theorem with boxes of Sarkar, Xue, and Sober\'on \cite{Sarkar2021}.  However, the proof presented here does not rely on the Brunn--Minkowski inequality, so it shows an alternate way to achieve this result.

\subsection{Proof using a Radon-style theorem}

\begin{definition}
For $X\subseteq\R^d$ let $\bbox(X)$ denote the intersection of all axis-parallel boxes which contain $X$ as a subset.
\end{definition}

The following lemma is reminiscent of Radon's theorem for convex hull.  Notice that the partition induced on the set always has one part as a singleton.

\begin{lemma}\label{lemma:Radon-box}
    Let $X\subseteq\R^d$ be a set such that $|X|\geq 2d+1$. Then their exists $x\in X$ such that $x\in \bbox(X\setminus\{x\})$.
\end{lemma}

\begin{proof}
    If $X$ is infinite, consider only a finite number of points of $X$. Let $a_i\in X$ be a point with the minimum $i$-th coordinate. Similarly, let $b_i$ be a point with the maximum $i$-th coordinate. Then $\bbox(X)=\bbox(\{a_1,b_1,\cdots,a_d,b_d\})$. As $|X|\geq 2d+1$, there exists $x\in X$ such that $x\notin\{a_1,b_1,\cdots,a_d,b_d\}$. Thus $x\in\bbox(X)=\bbox(\{a_1,b_1,\cdots,a_d,b_d\})\subseteq\bbox(X\setminus\{x\})$.
\end{proof}

\begin{remark}
    The set $\{ \pm e_1$ , $\dots,\pm e_d\}$ where $e_i$ is the standard basis vector in dimension $i$ shows that \cref{lemma:Radon-box} is sharp.
\end{remark}

The above lemma is essentially a discrete version of radon's lemma for boxes. In 2023, Breen defined a function $r(d)$ such that any set $X\in\R^d$ with $|X|\geq r(d)$, admits a partition into two sets $X_1,X_2$ such that $\bbox(X_1)\cap\bbox(X_2)\neq\emptyset$. Our above lemma is a discrete version of Breen's lemma. We show that $|X|\geq 2d+1$ is the minimum number of points required to guarantee a partition into two sets $X_1=\{x\},X_2=X\setminus\{x\}$ whose box hulls intersect and contain a point from the original set in their intersection. That is $\bbox(X_1)\cap \bbox(X_2)\cap X\neq\emptyset$.

Next we provide an alternate proof of Halman's Theorem. This proof is based off induction, and it closely mirrors Radon's proof of Helly's Theorem. This proof requires a finite family of boxes, but it also allows $S$ to be any subset of $\R^d$, not necessarily discrete.

\begin{theorem}
    Let $\mathcal{B}$ be a finite family of axis parallel boxes in $\R^d$, and consider any set $S\subset\R^d$. If every subfamily of $2d$ or fewer boxes in $\mathcal{B}$ contains a point of $S$ in its intersection, then $\left(\bigcap\mathcal{B}\right) \cap S\neq\emptyset$.
\end{theorem}

\begin{proof}
    If $|\mathcal{B}|\leq 2d$, then the statement is true. Suppose that the theorem holds for all families of boxes of size $n$ for some $n \ge 2d$. Consider a family of boxes $\mathcal{B}$ such that $|\mathcal{B}|= n+1$.
    By the inductive hypothesis, there must be a point of $S$ in the intersection of every $n$-tuple of boxes in $\mathcal{B}$. 
    
    For each $B_i \in \cal{B}$, let $s_i\in\left( \bigcap(\mathcal{B}\setminus B_i)\right)\cap S$. If $s_i=s_j$ for some $i\neq j$, then we are done. Suppose otherwise. Then $|\{s_1,\cdots,s_{n+1}\}|\geq 2d+1$. Hence by \cref{lemma:Radon-box}, there exists $s_k$ such that $s_k\in\bbox(\{s_i:i\neq k\})$. But $\{s_i:i\neq k\}\subseteq B_k$. Therefore we have $\bbox(\{s_i:i\neq k\})\subseteq B_k$ which implies $s_k\in B_k$. Recall by definition, that $s_k\in\left(\bigcap( \mathcal{B}\setminus B_k)\right)\cap S$, so $s_k\in\left(\bigcap\mathcal{B}\right)\cap S$.
\end{proof}

This shows that \cref{lemma:Radon-box} implies Halman's theorem.  The reverse implication also holds.

\begin{proposition}
    Lemma \ref{lemma:Radon-box} is a direct consequence of Halman's theorem.
\end{proposition}
\begin{proof}
 Consider a finite point set $X=\{x_1,\cdots,x_n\}$ with $n\geq 2d+1$. Let $B_i=\bbox(X\setminus\{x_i\})$. Let $\mathcal{B}=\{B_i:i\in[n]\}$. As each $B_i$ contains all but one of the points in $X$, any subfamily of $\mathcal{B}$ of size $2d$ must contain at least one point of $X$ because $|X|\geq 2d+1$. Therefore any subfamily of size $2d$ contains a point of $X$ in its intersection. Thus by Halman's Theorem $\cap\mathcal{B}$ contains a point $x_k\in X$ in its intersection. But $x_k\in \cap\mathcal{B}$ implies $x_k\in B_k=\bbox(X\setminus\{x_k\})$.
\end{proof}

The discrete nature of these results, which involves the set $S$, allows us to prove some versions which have no non-discrete analogue.  Consider the following ``very colorful'' quantitative Halman.  The key observation is that we can replace $S$ by several sets $S_1,\dots, S_m$, and we can prescribe how many points we want to contain of each set.

\begin{theorem}[Very colorful Halman]\label{thm:very-colorful-quantitative-halman}
    Let $\cal{B}_1,\cal{B}_2,\cdots,\cal{B}_{2d}$ be finite families of axis parallel boxes in $\R^d$. Let $m$ be a positive integer, Let $S_1, \dots, S_m \subseteq\rr^d$ be a discrete sets and $n_1, \dots, n_m$ be a positive integers. Suppose that for every choice of  $B_i\in\cal{B}_i$ for each $i \in [2d]$ we have $\left|\bigcap_{i=1}^{2d}B_i\cap S_j\right|\geq n_j$ for each $j \in [m]$.
    Then there exists an index $l\in[n]$ such that the intersection of $\cal{B}_l$ contains $n_j$ points of $S_j$ for all $j \in [m]$.  In other words, $|\cap\cal{B}_l\cap S_j|\geq n_j$.
\end{theorem}

A direct application of \cref{thm: colorful quantitative halman} for each $S_j$ gives us a similar result, but we have no guarantee that the family $\cal{B}_l$ will be the same for all values of $j$.

\begin{proof}
The proof follows the same idea as the proof of \cref{thm: colorful quantitative halman}.  Order the set $S_j$ lexicographically.  We identify the elements of $S$ with the integers in $[|S_j|]$.  We call a subfamily $\cal{F}\subseteq\cal{B}_1\cup\cdots\cup\cal{B}_{2d}$ colorful if no two boxes in $\cal{F}$ come from the same indexed set $\cal{B}_i$. For each $j \in [m]$ and $k \in [n_j]$, consider a function $f^j_k$ from the set of all colorful $(2d-1)$-tuples to the set $[|S_j|]$. The function $f^j_k$ assigns to each colorful $(2d-1)$-tuple the label of the $k$-th smallest point of $S_j$, under the lexicographic order, which is contained in the intersection of this tuple. As $S_j$ is finite and each colorful $(2d-1)$-tuple must contain at least $n_j$ points of $S_j$, the function $f^j_k$ is well defined.  

Since there is a finite number of possible $(2d-1)$-tuples, there is a $(2d-1)$-tuple $\cal{A}$ on which $f:=\sum_{j=1}^m f^j_{n_j}$ attains its maximal value. 

Without loss of generality, assume that $\cal{A}$ does not have a set of $\cal{B}_{2d}$ and $\cal{A}=\{A_1,\cdots,A_{2d-1}\}$ where $A_i\in\cal{B}_i$ for each $i \in [2d-1]$. Consider a set $B\in\cal{B}_{2d}$. We claim that the points in $S_j$ corresponding to $f^j_1(\cal{A}),\dots,f^j_{n_j}(\cal{A})$ are contained in $B$. The intersection $(\cap\cal{A})\cap B$ is also an axis-parallel box, so its projection to the $i$-th coordinate is an interval, which we denote by $[b_i, c_i]$. Let $B_i$ be a box in $\cal{A}\cup B$ which achieves $b_i$ as a left endpoint when projected to the $i$-th coordinate. Similarly, let $C_i$ be a box from this collection which achieves $c_i$ as a right endpoint when projected to the $i$-th coordinate. 
 
 Consider the subfamily $\cal{A}'=\{B_1,\dots,B_d,C_2,\cdots,C_d\}$ (note we are not using $C_1$). It may be the case that these sets are not all unique and hence $|\cal{A}'|<2d-1$. If this is the case, include additional arbitrary boxes from $\cal{A}\cup B$ into $\cal{A}'$ so that we have $|\cal{A}'|=2d-1$.  As in the proof of \cref{thm: colorful quantitative halman}, the value of $f^j_i$ in $\cal{A}'$ can only be larger than or equal to the one in $\cal{A}$.  If we didn't have the same initial $n_j$-tuple of points as in $\cal{A}$, we would contradict the maximaliy of $\cal{A}$.  This implies that $B$ contains the $n_j$ points of $S_j$ we need.
\end{proof}

\section{The topology of Halman's theorem}\label{sec:topological}

One approach to study Helly's theorem is to construct simplicial complexes associated to families of convex sets.  For Helly's theorem, the nerve complex of a family $\mathcal{F}$ of convex sets provides this link.  Given a family $\mathcal{F}$ of sets, this is a simplicial complex with a vertex for each set of $\mathcal{F}$ and a face for each set of vertices that corresponds to an intersecting subfamily.

For a face $\sigma$ in a simplicial complex $K$, we define the dimension of $\sigma$ as $|\sigma|-1$.

One of the key operations on simplicial complexes are collapses.  Given a simplicial complex $K$ and an integer $m$, an $m$-collapse is a modification of $K$ to a  new simplicial complex $K'$ such that there following properties are satisfied:
\begin{itemize}
    \item There is a face $\sigma$ of $K$ that is contained in a unique inclusion-maximal face $\eta$ of $K$.
    \item The face $\sigma$ has dimension at most $m-1$.
    \item The complex $K'$ is formed by removing $\sigma$ and all faces that contain $\sigma$.
    \[
    K' = K \setminus{\tau : \sigma \subset \tau}.
    \]
\end{itemize}

We say that a complex $K$ is $m$-collapsible if there is a sequence of $m$-collapses that starts with $K$ and ends with the empty complex.  A lot of properties of finite families of convex sets follow from the fact that the nerve complex of a finite family of convex sets in $\rr^d$ is $(d-1)$-collapsible \cite{Tancer:2013iz}.

Now, given a finite family $\mathcal{B}$ of axis-parallel boxes in $\rr^d$ and a finite set $S \subset \rr^d$, we can create a simplicial complex $K(\cal{B}, S)$ as follows:
\begin{itemize}
    \item We include a vertex in $K(\cal{B}, S)$ of each box $B \in \cal{B}$ that contains a point in $S$.
    \item A set of vertices of $K(\cal{B}, S)$ forms a face if an only if the intersection of their corresponding boxes contains a point of $S$.
\end{itemize}

\begin{theorem}\label{thm:collapse-halman}
    If $\mathcal{B}$ is a finite family of axis-parallel boxes in $\rr^d$ and $S \subset \rr^d$ is finite, then $K(\cal{B}, S)$ is $(2d-1)$-collapsible.
\end{theorem}

\begin{proof}
    We prove this by induction on $|S|$.  If $S = \emptyset$, then $K(\cal{B}, S) = \emptyset$, which is $(2d-2)$-collapsible.  Now assume that for each set $S'$ with $|S|-1$ elements, the complex $K(\cal{B}, S')$ is $(2d-1)$-collapsible.

    If there is a point $s \in S$ such that $s$ is not the unique element of $S$ in the intersection of some boxes of $\cal B$, then define $S'=S\setminus\{s\}$.  We have $K(\cal{B}, S)=K(\cal{B}, S')$ and we are done.
    
    Now assume that for each $s \in S'$, there is a family of boxes for which $s$ is the only point of $S$ in their intersection.  Let $s_0$ be the lexicographic maximum of $S$, and $\cal B _0 \subset \cal B$ be the set of boxes in $\cal B$ that contain $s_0$.  By construction, we know that $S \cap \left(\bigcap \cal B_0\right) =\{s_0\}$.  By the same arguments as in the proof of \cref{thm: colorful quantitative halman}, there exists a family $\cal B' \subset \cal B_0$ with at most $2d-1$ set such that the lexicographic minimum of $S \cap \left(\bigcap \cal B_0\right)$ and $S \cap \left(\bigcap \cal B'\right)$ are the same.  In this case, they can only be $\{s_0\}$.

    In other words, we have found a family $\cal B'$ with at most $(2d-1)$ whose intersection has a single point $s_0$ of $S$.  Let $\sigma$ be the corresponding face of $K(\cal{B}, S)$, of dimension at most $2d-2$.  Every face $\rho \supset \sigma$ must have points of $S$ in its intersection, and therefore its intersection with $S$ must be $\{s_0\}$.  This means that for the face $\eta$ corresponding to $\cal B_0$, we have $\eta \supset \rho \supset \sigma$.  In other words, $\eta$ is the unique inclusion-maximal face that contains $\sigma$.  We can do a collapse on $\sigma$, and obtain a new complex $K'$.
    
    The final observation is that $K' = K(\cal{B}, S\setminus\{s_0\})$, as every removed face had $s_0$ as their unique intersection with $S$.  By induction, $K'$ is $(2d-1)$-collapsible, so $K$ must be too.
\end{proof}

The reason collapsibility of our complexes is important is because $m$-collapsility of a complex gives bounds on its Leray number, and for complexes with bounded Leray numbers there are many Helly-type results known.  A complex $K$ is $m$-Leray if for every induced subcomplex $L$ of $K$, the reduced homology group $\tilde{H}_i(L)$ over $\mathds{Q}$ vanishes for $i \ge m$.  In particular, if $K$ is $m$-collapsible, then $K$ is $m$-Leray since each $m$-collapse does not affect $\tilde{H}_i(K)$ of $i \ge m$.

Therefore, we can apply directly the fractional Helly theorem for $m$-Leray complexes of Kalai \cite{Kalai:1984bg} and the colorful Helly theorem for $m$-Leray complexes of Kalai and Meshulam \cite{Kalai:2005tb}.  We include below the direct consequences of \cref{thm:collapse-halman}.

\begin{theorem}[Fractional Halman theorem for $2d$-tuples]\label{thm:fractional-strong}
    Let $d$ be a positive integer, $\alpha \in (0,1)$ and $\beta = 1- (1-\alpha)^{1/(2d-1)}$.  Let $\cal B$ be finite family of axis-parallel boxes in $\rr^d$ and $S \subset \rr^d$ be a finite set.  If at least $\alpha \binom{|\cal{B}|}{2d}$ of the $2d$-tuples of $\cal B$ contain a point of $S$ in their intersection, then there is a subfamily $\cal B' \subset \cal B$ of at least $\beta |\cal{B}|$ set of $\cal B$ whose intersection contains a point of $S$.
\end{theorem}

Note that as $\alpha \to 1$, we have $\beta \to 1$ as well.  This means that \cref{thm:fractional-strong} implies Halman's theorem.  For the theorem below, a matroid complex is a particular simplicial complex corresponding to the independent set of a matroid.

\begin{theorem}[Matroid colorful Halman]\label{thm:matroid-halman}
    Let $S$ be a finite set in $\rr^d$ and $M$ be a matroid complex with rank function $\rho$ and vertex set $V$.  For each $v$ in $V$, we are given an axis-parallel box $B_v$ in $\rr^d$.  For each independent set $U \subset V$, there is a point of $S$ in $\bigcap_{v \in U} B_v \subset \rr^d$.  Then, there exists a set $T \subset V$ such that $\rho (V \setminus T) \le 2d-1$ and such that there is a point of $S$ in $\bigcap_{v \in T} B_v$.
\end{theorem}

For example, we can consider $M$ a matroid complex on a set of vertices $V_1 \cup V_2 \cup \dots \cup V_{2d}$ and define the elements of $M$ as the sets $V$ with at most one element of $V_i$.  The rank $\rho$ of a set is just the number of sets $V_i$ it intersects.  For this matroid $M$, we recover \cref{thm: colorful quantitative halman} for $n = 1$.  The major upside is that we can use any other matroid $M$.

\section{Other fractional versions}\label{sec:fractional-and-beyond}

To prove \cref{thm:fractional-small}, we adapt the approach of B\'ar\'any and Matou\v{s}ek to Halman's theorem.  The main goal is that if we start with families $\cal B$ and $S$ that do not satisfy \cref{thm:fractional-small}, then we can construct a family $\cal B' \subset \cal B$ that contradicts the colorful Helly theorem (without involving $S$ any more).  We start with a technical definition and a lemma.

\begin{definition}
    Let $\varepsilon>0$ and $Z=Z_1\dot{\bigcup}\dots\dot{\bigcup} Z_r$ be a disjoint union of multisets (we allow repetition of elements). The multisets $Z_1,\cdots,Z_r$ are called $\varepsilon$-\textit{box-intermixed} if for every halfspace $H$ which is orthogonal to an element of the canonical basis, we have $|H\cap Z|\geq \varepsilon|Z|$ implies that $H\cap Z_j$ is nonempty for all $j\in[r]$.
\end{definition}

\begin{lemma}[Intermixing Box Lemma]\label{Lem: Box Intermix}
    Let $S_1,\cdots,S_r$ be finite point multisets in $\R^d$ that are $\frac{1}{2d}$-box-intermixed. Then 
    \[
    \bigcap_{j=1}^r\bbox(S_j)\cap S\neq\emptyset.
    \]
\end{lemma}

\begin{proof}
    Let $S=S_1\cup\cdots\cup S_r$. Define the following set of boxes $\cal{B}=\{\bbox(X): X\subseteq S\text{ and } |X|>\frac{2d-1}{2d}|S|\}$. Note that for an arbitrary set $\bbox(X)$ in this family we have $|S\setminus\bbox(X)|<\frac{|S|}{2d}$. The complements of any $2d$ sets in $\cal{B}$ cannot contain all the points of $S$, so any $2d$ sets of $\cal{B}$ contain a point of $S$ in their intersection.
    
    As the sets of $\cal{B}$ are closed axis parallel boxes, by \cref{thm:halman} there is a point $s\in S$ such that $s\in\cap\cal{B}$. Suppose $s\notin \bbox(S_i)$ for some $i$. Then there exists a hyperplane $H$ orthogonal to one of the canonical basis vectors which contains $s$ and separates $S$ from $\bbox(S_i)$.  
    
    Let $A$ be the closed halfspace generated by this hyperplane which contains $s$ and is disjoint from $\bbox(S_i)$. Then $|S\cap A|\ge\frac{|S|}{2d}$, otherwise $\bbox(S\setminus A)$ would be a box in $\cal{B}$ that does not contain $s$. Therefore, $A$ is a half-space containing at least $\frac{1}{2d}|S|$ points of $S$ and no points of $S_i$. This contradicts $S_1,\cdots,S_r$ being $(1/2d)$-box-intermixed.
\end{proof}

The following lemma is key to go from intersection properties of families of boxes involving $S$ to intersection properties of families of boxes that do not involve $S$.

\begin{lemma}[The intermixing lemma]\label{lem: intermixing consequence}
    Let $S_1,\cdots,S_r$ be finite sets of points in $\rr^d$ of the same cardinality and $S$ be another set of points in $\rr^d$.  Let $I$ be a set of indices used to label each $S_j$, so that $S_j=\{s_j^i:i\in I\}$. Then one of the following holds:
    \begin{enumerate}
        \item $\displaystyle\bigcap_{j=1}^n\bbox(S_j)\cap S\neq\emptyset$.
        \item There exists $I'\subseteq I$ and indices $m,n\in[r]$ such that $|I'|\geq \frac{|I|}{2d}$ and $\bbox(S'_m)\cap\bbox(S'_n)=\emptyset$. Here $S'_j=\{s_j^i:i\in I'\}$.
    \end{enumerate}
\end{lemma}

\begin{proof}
    If $S_1,\cdots,S_r$ are $\frac{1}{2d}$-box-intermixed, then $(1)$ follows immediately from Lemma \ref{Lem: Box Intermix}. If not, then there exists a hyperplane $H$ orthogonal to one of the canonical basis vectors such that one of the closed half-spaces $X$ formed by this hyperplane contains $\frac{r|I|}{2d}$ points from all the sets but no points from some $S_m$. Thus, there must be some index $n$ for which $X$ contains at least $\frac{r|I|}{(r-1)2d} \ge \frac{|I|}{2d}$ of the points of $S_n$. That is, $|X\cap S_n|\geq \frac{|I|}{2d}$. Let $I'=\{i\in I: s_n^i\in X\}$. As $H$ is orthogonal to one of the axes, we have $\bbox(S'_m)\cap\bbox(S'_n)=\emptyset$ as desired.
\end{proof}

Once \cref{lem: intermixing consequence} is established, the rest of the proof of \cref{thm:fractional-small} follows the structure of the B\'ar\'any--Matou\v{s}ek proof of their fractional Helly theorem for lattice convex sets.  The intermediary steps contain interesting results, which may also be interesting to the readers, so we include the details below.

To prove \cref{thm:fractional-small}, we need yet another colorful version of Halman's theorem.  In this one we will use only $d+1$ color classes.  Since the colorful Halman theorem is optimal (as Halman's theorem is optimal), the downside is that we won't be able to conclude that one color class has a point of $S$ in its intersection.  Instead, we will be able to guarantee that there are many sets in a single color class that have a point of $S$ in their intersection.

\begin{theorem}[Colorful Halman with few colors]\label{thm:colorful-halman-few-colors}
     Given integers $r,d$, there exists an integer $t$ such that the following holds. 
 Let $\cal{B}_1,\cal{B}_2,\cdots,\cal{B}_{d+1}$ be finite families of axis parallel boxes in $\R^d$, each with cardinality $t$. Let $S\subseteq\rr^d$ be a discrete set. Suppose that for every choice of  $B_i\in\cal{B}_i$ for each $i \in [d+1]$ we have $\bigcap_{i=1}^{d+1}B_i$ is not empty and has at least one point of $S$.
    Then there exists an index $l\in[d+1]$ such that there are $r$ boxes in $\cal{B}_l$ whose intersection contains at least one point of $S$.
\end{theorem}

We will only use \cref{thm:colorful-halman-few-colors} with $r=2d$, but we prove it for any $r$.  Since the colorful Halman theorem is optimal, we cannot expect $t=r$ if we use fewer than $2d$ colors.  We will use the following theorem by Erd\H{o}s and Simonovits \cite{Erdos1983}:

\begin{theorem}[Erd\H{o}s, Simonovits 1983]
    For every $\alpha >0$ and positive integers $d,t$, there exists a $\delta >0$ such that the following holds.  Let $H$ be a $(d+1)$-uniform hypergraph on $n$ vertices, with at least $\alpha \binom{n}{d+1}$ edges.  Then, $H$ contains at least $\delta n^{(d+1)t}$ copies of $K^{d+1}(t)$, the complete $(d+1)$-partite $(d+1)$-uniform hypergraph with each part having $t$ elements, as a subgraph.
\end{theorem}

\begin{proof}[Proof of \cref{thm:colorful-halman-few-colors}]
    Given a finite family of axis-parallel boxes $\cal B$ in $\rr^d$ and a finite set $S \subset \rr^d$, recall the definition of the simplicial complex $K(\cal B, S)$.  We consider the hypergraph $H_S(\cal{B},E)$ made of the faces of $K(\cal B, S)$ with $d+1$ vertices, and let $E$ be the set of hyperedges of this hypergraph.

    For $\cal{B} = \bigcup_{i=1}^{d+1}\cal{B}_i$, we have that $H_S(\cal{B},E)$ contains the complete $(d+1)$-uniform $(d+1)$-partite hypergraph $K^{d+1}(t)$, where each of the components has $t$ vertices (corresponding to some $\cal{B}_i$).  We say that the $i$-th class of $H_S(\cal{B},E)$ is the set of vertices corresponding to $\cal{B}_i$.

    For each edge $e \in E$, let $s_e$ be a point in $S$ in the intersections of the boxes corresponding to $e$.  For a vertex $v$, let $S_v =\{s_e : v \in e \in E\}$ and $G_v = \bbox (S_v) \cap S$.  Note that the box corresponding to $v$ contains all the points of $G_v$.  We want to prove that there is a collection $R$ of $r$ points of some class of $H_S(\cal{B},E)$ such that $\bigcap_{v\in R} G_v \neq \emptyset$.  If we don't find such a set $R$ of vertices, we are going to construct a sequence of sub-hypergraphs $H_S(\cal{B},E)=H_0 \supset H_1 \supset \dots \supset H_k \supset \dots$.  As we trim vertices and edges from our graphs, the sets $S_v, G_v$ get updated.  If at no point we find the collection $R$ of $r$ vertices, we aim to reach a set of axis-parallel boxes of the form $\bbox(S_v)$ that contradict the colorful Helly theorem.

    For any sub-hypergraph $H$ of $H_S(\cal{B},E)$, we can construct the sets $S_v, G_v$ as above, restricting to the edges of $H$.  We say that the $i$-th class of $H$ is $(r,2)$-disjoint if among any $r$ sets of the form $\bbox (S_v)$ there are two that are disjoint.

    \begin{lemma}\label{lem:trimming}
        Let $d,r, q$ be positive integers.  There exists an integer $Q = Q(d,r,q)$ such that the following holds.  Let $i \in [d+1]$ and $H = K^{d+1}(Q)$.  Assume that for each edge $e$ of $H$, a point $s_e \in S$ is given, and we construct the sets $S_v, G_v$ as above.  Then, at least one of the following two statements holds:
        \begin{itemize}
            \item there is a set $R$ of $r$ vertices of the $i$-th class of $H$ such that $\bigcap_{v\in R} G_v \neq \emptyset$;
            \item there is a subhypergraph of $H'$ of $H$ whose $i$-th class is $(r,2)$-disjoint.
        \end{itemize}
    \end{lemma}

    Establishing \cref{lem:trimming} finishes the proof of \cref{thm:colorful-halman-few-colors}.  Provided $t$ is large enough, as long as we don't find a class of vertices $R$ as in the first statement (in which case we would be done), we can start using \cref{lem:trimming} iteratively to try to make each class $(r,2)$-disjoint.  If every class is $(r,2)$-disjoint, we choose two vertices in each class corresponding to disjoint boxes of the form $\bbox(S_v)$ in $\rr^d$.  By construction, if we pick one box from each pair, they correspond to an edge $e$ of our hypergraph, so $s_e$ is contained in each of the corresponding boxes.  This construction contradicts Helly's colorful theorem.

    Now let's prove \cref{lem:trimming}.  Denote by $K^{d+1}(x_1, \dots, x_{d+1})$ the complete $(d+1)$-partite $(d+1)$-uniform hypergraph in which the $j$-th class has $x_j$ vertices.  We assume without loss of generality that $i= 1$.  Our starting hypergraph is $K^{d+1}(Q,Q,\dots, Q)$.  We pick any $q$ vertices of the first class and restrict ourselves to them, so we have a new hypergraph $H_0 = K^{d+1}(q,Q,Q,\dots, Q)$.  Denote by $V_1$ the $q$ vertices of the first class of $H_0$.  Let $R_1$ be an $r$-tuple of vertices of $V_1$.  
    
    Consider the $r$ (multi)sets $S_v$ for $v \in R_1$.  Note that the cardinality of each $S_v$ does not depend on $v$, as we get one point for each edge containing $v$.  In other words, the set $S_v$ is indexed by $K^d(Q)$ (call this set $I$).  We can apply \cref{lem: intermixing consequence} to them.  If item (1) holds, we have found our $r$ sets $G_v$ that intersect.  If the second happens, we can keep a subset $I' \subset I$ with at least $(1/2d)$ such that two of the corresponding boxes $\bbox(S'_v)$ are disjoint.  Now look at the subgraph of $K^d(Q)$ induced by this set $I'$ of edges.  Since it contains a $(1/2d)$-fraction of the edges of $K^d(Q)$, by the Erd\H{o}s-Simonovitz theorem, there is a number $q_1 = q_1(Q)$ such that there is a $K^d(q_1)$ induced in this subgraph.  Now, restrict $H_0 = K(q,Q, \dots, Q)$ to $H_1 = K(q,q_1,\dots, q_1)$, by restricting the $j$-th class of $H_0$ to the $(j+1)$-th class of this $K^d(q_1)$ for all $j \in [d]$.

    What we have achieved is that now, the $r$-tuple $R_1$ induces a pair of disjoint boxes.  We repeat this process $\binom{q}{r}$ times, once for each $r$-tuple of $V_1$.  If we never found an $r$-tuple of vertices of $V_1$ that induce an intersecting $r$ tuple of sets of the form $G_v$ for $v$, we get a sequence of hypegraphs $H_k = K^{d+1}(q,q_k, \dots, q_k)$ for a decreasing sequence $Q>q_1 > \dots > q_k > \dots$.  Provided that $Q$ is large enough, we will have $q_k \ge q$ for $k = \binom{q}{r}$.  At this point, we can assume without loss of generality that $q_k = q$ for $k = \binom{q}{r}$, and we have proven the lemma.
\end{proof}






\section{$H$-convex sets}\label{sec:hconvex}

If we want to generalize the discrete Helly-type theorems of this manuscript to a wider family of sets that axis-parallel boxes, we can use $H$-convex sets.  These sets have been studied before for their Helly-type properties \cites{Boltyanski:2003ir, Sarkar2021}.

\begin{definition}[$H$-convex set]
    Let $H\subseteq S^{d-1}$ be a set of vectors in the unit ball which are not contained in any closed half sphere. A set is $H$-convex if it can be written as the intersection of halfspaces of the form $H(h,c)=\{x\in\R^d:\langle x,h\rangle\leq c\}$ where $h\in H$ and $c\in\R$.
\end{definition}

\begin{figure}
    \centering
    \includegraphics[width=\textwidth]{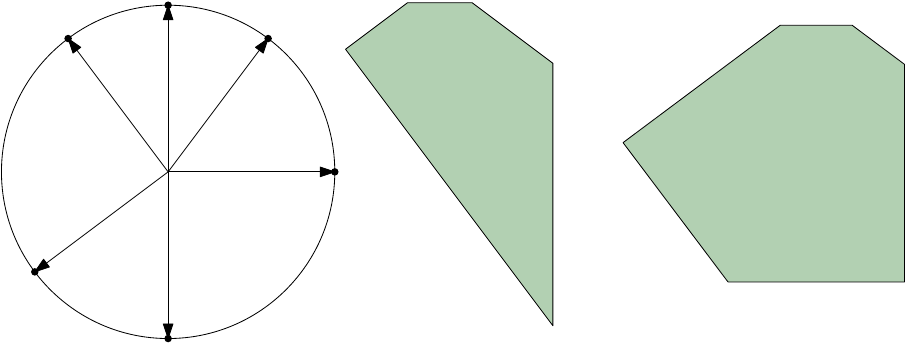}
    \caption{An example of a set $H$ and two $H$-convex sets.  Note that an $H$ convex set does not need to have a facet for each element of $H$.}
    \label{fig:hconvex}
\end{figure}

We present an illustration of $H$-convex sets in \cref{fig:hconvex}.  For any fixed $H$, the intersection of finite number of $H$-convex sets is also a $H$-convex set. This key observation allows us to generalize the proofs of the previous sections to this class of polytopes. The case of axis parallel boxes can be recovered by letting $H=\{\pm e_1,\cdots,\pm e_d\}$ where $e_1,\cdots,e_d$ are the elements of the canonical basis of $\rr^d$.

Let us state a Halman-type theorem for $H$-convex sets.

\begin{theorem}\label{thm:discrete-helly-hconvex}
    Let $H$ be a finite set of unit direction vectors in $\rr^d$ which are not contained in any closed half unit sphere. Let $k=|H|$ and $S\subset \rr^d$ be a discrete point set. Let $\cal{F}$ be a finite family of $H$-convex sets such that $|\cal{F}|\ge k$. If the intersection of any $k$ sets of $\cal{F}$ contains a point of $S$ then $\cap\cal{F}$ contains a point of $S$.
\end{theorem}

\begin{proof}
    Assume without loss of generality that $-e_1 \in H$, and order the points of $S$ lexicographically.  For each $(k-1)$-tuple $\cal{A}$ of sets in $\cal{F}$, consider their lexicographic maximum $s \in S$ in $\bigcap \cal{A}$.  Let $\cal{A}_0$ be the $(k-1)$-tuple whose lexicographic maximum $s_0$ is minimal.
    
    Let $B$ be any other set in $\cal{F}$.  We will show that $s_0 \in B$.  Assume for the sake a contradiction that $B$ does not contain $s_0$.  Then, notice that $R=B \cap \left( \bigcap \cal{A}_0\right)$ is an $H$-convex set.  Each facet of $R$ is determined by a hyperplane.  For each such hyperplane $\pi$, there must be a set in $\{B\} \cup \cal{A}_0$ that has $\pi$ as a support hyperplane.  Let $\cal{B} \subset \{B\} \cup \cal{A}_0$ be the family of sets constructed this way for each facet of $R$, except the facet in direction $-e_1$ (if there is one).  Therefore, $\cal{B}$ has at most $k-1$ sets.  If $\cal{B}$ has fewer than $k-1$, we can complete it to a $(k-1)$-tuple using sets of $\{B\} \cup \cal{A}_0$.  Notice that all points in $\bigcap \cal{B}$ are in $R \subset \bigcap \cal{A}_0$ or have smaller first coordinate.  Since $R$ must have points of $S$, it means that the lexicographic maximum of $\bigcap B$ is smaller than $s_0$, giving us the contradiction wanted.

    As every other set $\cal B$ must contain $s_0$, we have the desired result.
\end{proof}

Another interesting example of $H$-convex sets is when $H = d+1$.  In this case, all $H$-convex sets are homothetic simplices.

Note that the argument of the proof above is identical to the one in the proof of \cref{thm: colorful quantitative halman}, so a colorful quantitative version of \cref{thm:discrete-helly-hconvex} follows, which we include below.

\begin{theorem}
     Let $H$ be a finite set of unit direction vectors in $\rr^d$ which are not contained in any closed half unit sphere. Let $k=|H|$, $n$ be a positive integer, and $S\subset \rr^d$ be a discrete point set.  Let $\cal{F}_1, \dots, \cal{F}_k$ be finite families of $H$ convex sets such that the intersection of every $k$-tuple $F_1 \in \cal{F}_1, \dots, F_k \in \cal{F}_k$ contains at least $n$ points of $S$.  Then, there exists $i \in [n]$ such that $\bigcap \cal{F}_i$ contains at least $n$ points of $S$.
\end{theorem}

Just as the quantitative Halman theorem implied a version for measure, the case $\cal{F}_1 = \dots = \cal{F}_d$ implies a quantitative version of \cref{thm:discrete-helly-hconvex} for any measure.  This improves earlier quantitative results for $H$-convex sets, as the methods only worked for log-concave measures \cite{Sarkar2021}*{Theorem 2.3.2}.

The same construction of discrete nerve complexes as in \cref{sec:topological} shows that the corresponding discrete nerve complexes for $H$-convex sets are $(|H|-1)$-collapsible.  Therefore, we obtain for free colorful matroid versions of \cref{thm:discrete-helly-hconvex} and fractional versions of \cref{thm:discrete-helly-hconvex}.

\begin{theorem}
     Let $H$ be a finite set of unit direction vectors in $\rr^d$ which are not contained in any closed half unit sphere and $k = |H|$.  Let $S$ be a finite set in $\rr^d$ and $M$ be a matroid complex with rank function $\rho$ and vertex set $V$.  For each $v$ in $M$, we are given an $H$-convex set $K_v$ in $\rr^d$.  For each independent set $U \subset V$, there is a point of $S$ in $\bigcap_{v \in U} K_v \subset \rr^d$.  Then, there exists a set $T \subset V$ such that $\rho (V \setminus T) \le k-1$ and such that there is a point of $S$ in $\bigcap_{v \in T} K_v$.
\end{theorem}

Just as the methods of B\'ar\'any and Matou\v{s}ek carry over to axis-parallel boxes, if we apply the arguments of \cref{sec:fractional-and-beyond} to $H$-convex sets we prove the following result.

\begin{theorem}\label{thm:fractional-small-hconvex}
    Let $H$ be a finite set of unit direction vectors in $\rr^d$ which are not contained in any closed half unit sphere and $k = |H|$.  Let $d$ be a positive integer and $\alpha \in (0,1)$.  There exists $\beta = \beta(\alpha, d, k) > 0$ such that the following holds.  If $S$ is a finite set in $\rr^d$ and $\cal B$ is a finite family of $H$-convex sets such that $\alpha \binom{|\cal B|}{d+1}$ of the $(d+1)$-tuples of $\cal{B}$ satisfy that their intersection contains points of $S$, then there is a subfamily of at least $\beta |\cal B|$ sets of $\cal B$ whose intersection contains a point of $S$.
\end{theorem}

In this case, even though we can make the Helly number for $H$-convex sets as large as we want, the fractional version only requires us to verify $(d+1)$-tuples.


\end{document}